\documentclass[10pt]{amsart}
\usepackage{graphicx}
\usepackage{amsthm,hyperref}
\newtheorem{theorem}{Theorem}[section]
\newtheorem{lemma}[theorem]{Lemma}
\theoremstyle{definition}
\newtheorem{definition}[theorem]{Definition}

\theoremstyle{remark}
\newtheorem{remark}[theorem]{Remark}
\numberwithin{equation}{section}

\begin{document}
\title[Dirichlet boundary value problem with discontinuous nonlinearity]{Dirichlet boundary value problem related to the $p(x)-$Laplacian with discontinuous nonlinearity}
\author{M. AIT HAMMOU}
\address[M. AIT HAMMOU]{Sidi Mohamed Ben Abdellah university, Laboratory LAMA, Department of Mathematics, Fez, Morocco}
%\curraddr{Sidi Mohamed Ben Abdellah university, Department of Mathematics, P.C 1796, Fez, Morocco, Tel +212662410828}
\email[M. AIT HAMMOU]{mustapha.aithammou@usmba.ac.ma}
\subjclass[2010]{Set-valued operators, Nonlinear elliptic equation, $p(x)-$Laplaciane, Sobolev spaces with variable exponent, Degree theory.}
%\date{June 1, 2017, accepted December 7, 2017.}

\begin{abstract}
In this paper, we prove the existence of a weak solution for the Dirichlet boundary value problem related to the $p(x)-$Laplacian
$$
-\mbox{div}(|\nabla u|^{p(x)-2}\nabla u)+u\in -[\underline{g}(x,u),\overline{g}(x,u)],
$$
by using the degree theory after turning the problem into a Hammerstein equation. The right hand side $g$ is a possibly discontinuous function in the second variable satisfying some non-standard growth conditions.
\end{abstract}
\maketitle
\section{Introduction}
%\label{intro}
The Laplacian $p(x)$-Laplacian has been widely used in the modeling of several physical phenomena. we can refer to \cite{R} for its use for electrorheological fluids and to \cite{AMSo,CLR} for its use for image processing. Up to these days, many results have been achieved
obtained for solutions to equations related to this operator. \par
We consider  the following nonlinear elliptic boundary value problem
\begin{equation}\label{Pr1}
\left\{\begin{array}{lll}
-\Delta_{p(x)}u+u\in -[\underline{g}(x,u),\overline{g}(x,u)] & \mbox{in }\; \Omega,\\
u=0 &\mbox{on }\; \partial\Omega.
\end{array}\right.
\end{equation}
where $-\Delta_{p(x)}u:=-div(|\nabla u|^{p(x)-2}\nabla u)$ is the $p(x)-$Laplacian, $\Omega\subset \mathbb{R}^N$ is a bounded domain, $p(\cdot)$ is a log-H\"older continuous exponent and $g$ is a possibly discontinuous function in the second variable satisfying some non-standard growth conditions.\par
By using  the degree theory for $p(\cdot)\equiv p$ with values in $(2,N)$, Kim studied in (\cite{K}) this problem after developing a topological degree theory for a class of locally bounded weakly upper semicontinuous set-valued operators of generalized $(S_+)$ type in real reflexive separable Banach spaces, based on the Berkovits-Tienari degree \cite{BT}.\par
The aim of this paper is to prove the existence of at least one weak solution for \eqref{Pr1} using the topological degree theory after
turn the problem into a Hammerstein equation. The results in \cite{K} are thus extended into a larger functional framework, that of Sobolev spaces with variable exponents.\par
This paper is divided into four sections The second section is reserved for a mathematical preliminaries concerning some
class of locally bounded weakly upper semicontinuous set-valued operators of generalized $(S_+)$ type, the topological degree developing by Kim \cite{K} and some basic properties of generalized Lebesgue-Sobolev spaces $W_0^{1,p(x)}$. In the third, we find some assumptions and technical lemmas. The Fourth section is reserved to state and prove the existence results of weak solutions of problem \eqref{Pr1}.
\section{Mathematical preliminaries}
%\label{sec:1}
%Text with citations \cite{RefB} and \cite{RefJ}.
\subsection{Some classes of operators and topological degree}
%\label{sec:2}
Let $X$ and $Y$ be two real Banach spaces and $\Omega$ a nonempty subset of $X$. A set-valued operator\\
$F:\Omega\subset X \rightarrow 2^Y$ is said to be
\begin{itemize}
  \item {\it bounded}, if it takes any bounded set into a bounded set;
  \item {\it upper semicontinuous (u.s.c)}, if the set $F^{-1}(A)=\{u\in\Omega/ Fu\cap A\neq\emptyset\}$ is closed in $X$ for each closed set $A$ in $Y$;
  \item {\it weakly upper semicontinuous (w.u.s.c)}, if $F^{-1}(A)$ is closed in $X$ for each weakly closed set $A$ in $Y$;
  \item {\it compact}, if it is u.s.c and the image of any bounded set is relatively compact.
  \item {\it locally bounded}, if for each $u\in \Omega$ there exists a neighborhood $\mathcal{U}$ of $u$ such that the set $F(\mathcal{U})=\bigcup_{u\in\mathcal{U}}Fu$ is bounded.
\end{itemize}

Let $X$ be a real reflexive Banach space with dual $X^*$. A set-valued operator\\
$F:\Omega\subset X\rightarrow 2^{X^*}\backslash\emptyset$ is said to be
\begin{itemize}
  \item {\it of class $(S_+)$}, if for any sequence $(u_n)$ in $\Omega$ and any sequence $(w_n)$ in $X^*$ with $w_n\in Fu_n$ such that $u_n \rightharpoonup u$ in $X$ and $limsup\langle w_n,u_n - u\rangle\leq 0 $, it follows that $u_n \rightarrow u$ in $X$;
  \item {\it quasimonotone}, if for any sequence $(u_n)$ in $\Omega$ and any sequence $(w_n)$ in $X^*$ with $w_n\in Fu_n$ such that $u_n \rightharpoonup u$ in $X$, it follows that $$liminf\langle w_n,u_n - u\rangle\geq 0.$$
\end{itemize}

Let $T:\Omega_1\subset X\rightarrow X^*$ be a bounded operator such that $\Omega\subset\Omega_1$. A set-valued operatopr $F:\Omega\subset X\rightarrow 2^X\backslash\emptyset$ is said to of {\it class $(S_+)_T$}, if for any sequence $(u_n)$ in $\Omega$ and any sequence $(w_n)$ in $X$ with $w_n\in Fu_n$ such that $u_n \rightharpoonup u$ in $X$, $y_n:=Tu_n\rightharpoonup y$ in $X^*$ and $limsup\langle w_n,y_n-y\rangle\leq 0$, we have $u_n\rightarrow u$ in $X$.

For any set $\Omega\subset X$ with $\Omega\subset D_F$, where $D_F$ denotes the domain of $F$, and any bounded operator $T:\Omega\rightarrow X^*$, we consider the following classes of operators:
\begin{eqnarray*}
  \mathcal{F}_1(\Omega) &:=& \{F:\Omega\rightarrow X^*\mid F \mbox{ is bounded, continuous and of class }(S_+)\}, \\
  \mathcal{F}_T(\Omega) &:=& \{F:\Omega\rightarrow 2^X\mid F \mbox{ is locally bounded, w.u.s.c and of class }(S_+)_T\}.
\end{eqnarray*}

Let $\mathcal{O}$ be the collection of all bounded open set in $X$. Define
  $$\mathcal{F}(X) := \{F \in \mathcal{F}_T(\bar{G})\mid G\in\mathcal{O}, T\in \mathcal{F}_1(\bar{G})\},$$
Here, $T\in \mathcal{F}_1(\bar{G})$ is called an \it essential inner map \rm to $F$.
 \begin{lemma}\label{l2.1}\cite[Lemma 1.4]{K}
 Let $G$ be a bounded open set in a real reflexive Banach space $X$. Suppose that $T\in \mathcal{F}_1(\bar{G})$ and $S:D_S\subset X^*\rightarrow 2^X$ is locally bounded and w.u.s.c such that $T(\bar{G})\subset D_S$. Then the following statements hold:
 \begin{enumerate}
   \item If $S$ is quasimonotone, then $I+ST\in\mathcal{F}_T(\bar{G})$, where $I$ denotes the identity operator.
   \item If $S$ is of class $(S_+)$, then $ST\in\mathcal{F}_T(\bar{G})$.
 \end{enumerate}
 \end{lemma}

 \begin{definition}
   Let $G$ be a bounded open subset of a real reflexive Banach space $X$, $T:\bar{G}\rightarrow X^*$ be bounded and continuous and let $F$ and $S$ be bounded and of class $(S_+)_T$. The affine homotopy $H:[0,1]\times\bar{G}\rightarrow 2^X$ defined by
   $$H(t,u):=(1-t)Fu+tSu \mbox{ for } (t,u)\in[0,1]\times\bar{G}$$
   is called an affine homotopy with the common essential inner map $T$.
 \end{definition}

 \begin{remark}\cite[Lemma 1.6]{K}
    The above affine homotopy satisfies condition $(S_+)_T$.
 \end{remark}
As in\cite{K}, we introduce a suitable topological degree for the class $\mathcal{F}(X)$:
 \begin{theorem}\label{t2.1}\cite[Definition 2.9 and Theorem 2.10]{K}
   Let $$\mathcal{M} =\{(F,G,h)|G\in \mathcal{O}, T\in\mathcal{F}_1(\bar{G}), F\in\mathcal{F}_T(\bar{G}), h\notin F(\partial G)\}.$$ There exists a unique degree
   function $d:\mathcal{M}\rightarrow \mathbb{Z}$  that satisfies the following properties:
   \begin{enumerate}
               \item (Existence) if $d(F,G,h)\neq 0$ , then the inclusion $h\in Fu$ has a solution in $G$,
               \item (Additivity) If $G_1$ and $G_2$ are two disjoint open subset of $G$ such that\\
               $h\not\in F(\bar{G}\setminus (G_1\cup G_2))$, then we have $$d(F,G,h)=d(F,G_1,h)+d(F,G_2,h),$$
               \item (Homotopy invariance) Suppose that $H: [0,1]\times \bar{G}\rightarrow X$ is a locally bounded w.u.s.c affine homotopy of class $(S_+)_T$ with the common essential inner map $T$. If $h:[0,1]\rightarrow X$ is a continuous curve in $X$ such that $h(t)\notin H(t,\partial G)$ for all $t\in [0,1]$, then the value of $d(H(t,.),G,h(t))$ is constant for all $t\in[0,1]$,
               \item (Normalization) For any $h\in G$, we have $$d(I,G,h)=1.$$
             \end{enumerate}
 \end{theorem}
 %\begin{remark}\cite[Definition 3.3]{KH}
 %  The above degree is defined as follows: $$d(F,G,h):=d_B(F|_{\bar{G_0}},G_0,h),$$  where $d_B$ is the Berkovits degree \cite{Ber} and $G_0$ is any open subset of $G$ with \\$F^{-1}(h)\subset G_0$ and $F$ is bounded on $\bar{G_0}$.
% \end{remark}
\subsection{The spaces $W_0^{1,p(x)}(\Omega)$}

We introduce the setting of our problem with some auxiliary results of the variable exponent Lebesgue and Sobolev spaces $L^{p(x)}(\Omega)$ and $W_0^{1,p(x)}(\Omega)$. For convenience, we only recall some basic facts with will be used later, we refer to \cite{FZ1,KR,ZQF} for more details.\\
Let $\Omega$ be an open bounded subset of $\mathbb{R}^N$, $N\geq2,$ with a Lipschitz boundary denoted by $\partial\Omega$. Denote $$C_+(\bar{\Omega})=\{h\in C(\bar{\Omega})|\inf_{x\in\bar{\Omega}}h(x)>1\}.$$
For any $h\in C_+(\bar{\Omega})$, we define $$h^+:=max\{h(x),x\in\bar{\Omega}\}, h^-:=min\{h(x),x\in\bar{\Omega}\}.$$
For any $p\in C_+(\bar{\Omega})$ we define the variable exponent Lebesgue space
$$L^{p(x)}(\Omega)=\{u;\ u:\Omega \rightarrow \mathbb{R} \mbox{  is measurable and }
\int_\Omega| u(x)|^{p(x)}\;dx<+\infty \}$$ endowed with {\it Luxemburg norm}
 $$\|u\|_{p(x)}=\inf\{{\lambda>0}/\rho_{p(x)}(\frac{u}{\lambda})\leq 1\}.$$ where $$\rho_{p(x)}(u)=\int_\Omega|u(x)|^{p(x)}\;dx, \;\;\; \forall u\in L^{p(x)}(\Omega). $$
$(L^{p(x)}(\Omega),\|\cdot\|_{p(x)})$ is a Banach space \cite[Theorem 2.5]{KR}, separable and reflexive \cite[Corollary 2.7]{KR}. Its conjugate space is $L^{p'(x)}(\Omega)$ where $\frac{1}{p(x)}+\frac{1}{p'(x})=1$ for all $x\in \Omega.$ For any $u\in L^{p(x)}(\Omega)$ and $v\in L^{p'(x)}(\Omega)$, H\"older inequality holds \cite[Theorem 2.1]{KR}
\begin{equation}\label{hol}
  \left|\int_\Omega uv\;dx\right|\leq\left(\frac{1}{p^-}+\frac{1}{{p^{'}}^-}\right)\|u\|_{p(x)}\|v\|_{p'(x)}\leq 2 \|u\|_{p(x)}\|v\|_{p'(x)}.
\end{equation}
 Notice that if $u\in L^{p(.)}(\Omega)$  then the following relations hold true (see \cite{FZ1})
\begin{equation}\label{imp0}
  \|u\|_{p(x)}<1(=1;>1)\;\;\;\Leftrightarrow \;\;\;\rho_{p(x)}(u)<1(=1;>1),
\end{equation}
\begin{equation}\label{imp1}
  \|u\|_{p(x)}>1\;\;\;\Rightarrow\;\;\;\|u\|_{p(x)}^{p^-}\leq\rho_{p(x)}(u)\leq\|u\|_{p(x)}^{p^+},
\end{equation}
\begin{equation}\label{imp2}
\|u\|_{p(x)}<1\;\;\;\Rightarrow\;\;\;\|u\|_{p(x)}^{p^+}\leq\rho_{p(x)}(u)\leq\|u\|_{p(x)}^{p^-},
\end{equation}
%\begin{equation}\label{e3}
%  \lim_{n\rightarrow\infty}\|u_n-u\|_{p(x)}=0 \;\;\;\Leftrightarrow\;\;\;\lim_{n\rightarrow\infty}\rho_{p(x)}(u_n-u)=0.
%\end{equation}
From (\ref{imp1}) and (\ref{imp2}), we can deduce the inequalities
\begin{equation}\label{e1}
  \|u\|_{p(x)}\leq \rho_{p(x)}(u)+1,
\end{equation}
\begin{equation}\label{ineq1}
  \rho_{p(x)}(u)\leq\|u\|_{p(x)}^{p^-}+\|u\|
  _{p(x)}^{p^+}.
\end{equation}
 If $p_1,p_2\in C_+(\bar{\Omega}), p_1(x) \leq p_2(x)$ for any $x\in\bar{\Omega},$ then there exists the continuous embedding
$L^{p_2(x)}(\Omega)\hookrightarrow L^{p_1(x)}(\Omega).$\\
Next, we define the variable exponent Sobolev space $W^{1,p(x)}(\Omega)$ as
 $$W^{1,p(x)}(\Omega)=\{u\in L^{p(x)}(\Omega)/|\nabla u | \in L^{p(x)}(\Omega)\}.$$ It is a Banach space under the norm
$$||u||=\|u\|_{p(x)}+\|\nabla u\|_{p(x)}.$$ We also define $W_0^{1,p(.)}(\Omega)$ as the subspace of $W^{1,p(.)}(\Omega)$ which is the closure of $C_0^\infty(\Omega)$ with respect to the norm $||\;.\;||$. If the exponent $p(.)$ satisfies the log-H\"older continuity condition, i.e. there is a constant $\alpha>0$ such that for every $x,y\in\Omega, x\neq y$ with $|x-y|\leq\frac{1}{2}$ one has
\begin{equation}\label{e4}
  |p(x)-p(y)|\leq\frac{\alpha}{-\log|x-y|}\;\;,
\end{equation}
then we have the Poincar\'e inequality (see \cite{HHKV,SD}), i.e. the exists a constant $C>0$ depending only on $\Omega$ and the function $p$ such that
\begin{equation}\label{ptcar}
  \|u\|_{p(x)}\leq C\|\nabla u\|_{p(x)} , \forall u\in W_0^{1,p(.)}(\Omega).
\end{equation}
In particular, the space $W_0^{1,p(.)}(\Omega)$ has a norm $\|\cdot\|$ given by
$$\|u\|_{1,p(x)}=\|\nabla u\|_{p(.)} \mbox{ for all } u\in W_0^{1,p(x)}(\Omega),$$ which is equivalent to $ \|\cdot\|.$  In addition, we have the compact embedding\\ $W_0^{1,p(.)}(\Omega)\hookrightarrow L^{p(.)}(\Omega)$(see \cite{KR}).
The space $(W_0^{1,p(x)}(\Omega),\|\cdot\|_{1,p(x)})$ is a Banach space, separable and reflexive (see \cite{FZ1,KR}). The dual space of $W_0^{1,p(x)}(\Omega),$ denoted  $W^{-1,p'(x)}(\Omega),$ is equipped with the norm
$$\|v\|_{-1,p'(x)}=\inf\{\|v_0\|_{p'(x)}+\sum_{i=1}^{N}\|v_i\|_{p'(x)}\},$$ where the infinimum is taken on all possible decompositions $v=v_0-div F$ with $v_0\in L^{p'(x)}(\Omega)$ and $F=(v_1,...,v_N)\in (L^{p'(x)}(\Omega))^N.$

\section{Basic assumptions and technical Lemmas}

In this section, we study the Dirichlet boundary value problem (\ref{Pr1}) with discontinuous nonlinearity, based on the degree theory in Section 2, where $\Omega\subset \mathbb{R}^N$, $N\geq2,$  is a bounded domain with a Lipschitz boundary $\partial\Omega$,  $p\in C_+(\bar{\Omega})$ satisfy the log-H\"older continuity condition (\ref{e4}), $2\leq p^-\leq p(x)\leq p^+<\infty$ and $g:\Omega\times\mathbb{R}\rightarrow \mathbb{R}$ is a possibly discontinuous real-valued function in the sense that
$$\underline{g}(x,s)=\liminf_{\eta\rightarrow s}g(x,\eta)=\lim_{\delta\rightarrow 0^+}\inf_{|\eta-s|<\delta}g(x,\eta),$$
$$\overline{g}(x,s)=\limsup_{\eta\rightarrow s}g(x,\eta)=\lim_{\delta\rightarrow 0^+}\sup_{|\eta-s|<\delta}g(x,\eta).$$
Suppose that
\begin{description}
  \item[($g_1$)] $\overline{g}$ and $\underline{g}$ are superpositionally measurable, that is, $\overline{g}(\cdot,u(\cdot))$ and $\underline{g}(\cdot,u(\cdot))$ are measurable on $\Omega$ for any measurable function $u:\Omega\rightarrow\mathbb{R};$
  \item[($g_2$)] $g$ satisfies the growth condition $$|g(x,s)|\leq k(x)+c|s|^{q(x)-1}$$ for a.e. $x\in\Omega$ and all $s\in\mathbb{R}$, where $c$ is a positive constant, $k\in L^{p'(x)}(\Omega)$ and $p\in C_+(\bar{\Omega})$ with $q^+ < p^-.$
\end{description}

\begin{lemma}\cite[Proposition 1]{C}\label{usc}
For each fixed $x\in \Omega$, the functions $\overline{g}(x,s)$ and $\underline{g}(x,s)$ are u.s.c functions on $\mathbb{R}^N$.
\end{lemma}
\begin{lemma}\cite[Theorem 3.1]{Ch}\label{L}
  The operator $L:W_0^{1,p(x)}(\Omega)\rightarrow W^{-1,p'(x)}(\Omega)$ setting by
  $$\langle Lu,v\rangle=\int_\Omega|\nabla u|^{p(x)-2}\nabla u\nabla v dx, \mbox{ for all } u,v\in W_0^{1,p(x)}(\Omega)$$
  is continuous, bounded and strictly monotone. It is also a homeomorphism mapping of class $(S_+)$.
\end{lemma}
\begin{lemma}\label{A}
  The operator $A:W_0^{1,p(x)}(\Omega)\rightarrow W^{-1,p'(x)}(\Omega)$ setting by
$$\langle Au,v\rangle=-\int_\Omega  uv\;dx \mbox{ for } u,v\in W_0^{1,p(x)}(\Omega)$$
is compact.
\end{lemma}
\begin{proof}
  Since $p(x)\geq 2$, we have $p'(x)\leq 2\leq p(x)$, then the embedding\\$i:L^{p(x)}\rightarrow L^{p'(x)}$ is continuous. Since the embedding $I:W_0^{1,p(x)}(\Omega)\rightarrow L^{p(x)}(\Omega)$ is compact, it is known that the adjoint operator $I^*:L^{p'}(\Omega)\rightarrow W^{-1,p'(x)}(\Omega)$ is also compact. Therefor, $A=I^*oioI$ is compact.
\end{proof}
\begin{lemma}\label{N}
  Under assumptions $(g_1)$ and $(g_2)$, the set-valued operator\\
$N:W_0^{1,p(x)}(\Omega)\rightarrow 2^{W^{-1,p'(x)}(\Omega)}$ setting by
  $$Nu=\{z\in W^{-1,p'(x)}(\Omega)| \exists w\in L^{p'(x)}(\Omega); \underline{g}(x,u(x))\leq w(x)\leq \overline{g}(x,u(x)) \mbox{ a.e. } x\in\Omega$$
   $$\mbox{ and } \langle z ,v\rangle=\int_\Omega wv dx,\;\;\ \forall v\in W_0^{1,p(x)}(\Omega)\}$$ is bounded, u.s.c and compact.
    %and $Nu$ is nonempty, closed and convex for every $u\in W_0^{1,p(x)}(\Omega).$
\end{lemma}
\begin{proof}
Let $\phi:L^{p(x)}(\Omega)\rightarrow2^{L^{p'(x)}(\Omega)}$ be the set-valued operator given by
$$\phi u=\{w\in L^{p'(x)}(\Omega)|; \underline{g}(x,u(x))\leq w(x)\leq \overline{g}(x,u(x)) \mbox{ a.e. } x\in\Omega\}.$$
  For each $u\in W_0^{1,p(x)}(\Omega)$, we have from the growth condition $(g_2)$
  $$max\{|\underline{g}(x,s)|,|\overline{g}(x,s)|\}\leq k(x)+c|s|^{q(x)-1}.$$
and by the inequalities \eqref{e1} and \eqref{ineq1}, it follows that
\begin{eqnarray*}
  \|\underline{g}(x,u(x))\|_{p'(x)} &\leq& \rho_{p'(x)}(\underline{g}(x,u(x)))+1 \\
   &=& \int_\Omega|\underline{g}(x,u(x))|^{p'(x)}+1 \\
  % &\leq& C\int_\Omega |k(x)+|u|^{p(x)-1}+|\nabla u|^{p(x)-1}|^{p'(x)}+1\\
   &\leq& 2^{p'^+}(\rho_{p'(x)}(k)+\rho_{r(x)}(u)+1 \\
   &\leq& 2^{p'^+}(\rho_{p'(x)}(k)+\|u\|_{r(x)}^{r^+}+\|u\|_{r(x)}^{r^-})+1,
\end{eqnarray*}
where $r(x)=(q(x)-1)p'(x)<p(x)$. By the continuous embedding $L^{p(x)}\hookrightarrow L^{r(x)}$, we have
$$\|\underline{g}(x,u(x))\|_{p'(x)}\leq 2^{p'^+}(\rho_{p'(x)}(k)+\|u\|_{1,p(x)}^{r^+}+\|u\|_{1,p(x)}^{r^-})+1.$$
A similar inequality holds for $\overline{g}(x,u(x))$, so that $\phi$ is bounded on $W_0^{1,p(x)}(\Omega)$.\\
Let's show that $\phi$ is u.s.c, i.e.,
 $$\forall\varepsilon>0, \exists\delta>0; \parallel u-u_0\parallel_{p(x)}<\delta\Rightarrow\phi u\subset\phi u_0+B_\varepsilon,$$
where $B_\varepsilon$ is the $\varepsilon$-ball in $L^{p'(x)}(\Omega)$.\\
To this end, given $u_0\in L^{p(x)}(\Omega)$, we consider the point sets
$$E_{m,\varepsilon}=\bigcap_{t\in \mathbb{R}}G_t$$
where
$$G_t=\{x\in \Omega; |t-u_0(x)|<\frac{1}{m}\Rightarrow[\underline{g}(x,t),\overline{g}(x,t)]\subset(\underline{g}(x,u_0(x)-\frac{\varepsilon}{R},\overline{g}(x,u_0(x))+
\frac{\varepsilon}{R});$$
$m$ being an integer and $R$ being a constant to be determined.\\
It is obvious that
$$E_{1,\varepsilon}\subset E_{2,\varepsilon}\subset...$$
By Lemma \ref{usc},
$$\bigcup_{m=1}^\infty E_{m,\varepsilon}=\Omega,$$
thus there is an integer $m_0$ such that
\begin{equation}\label{1.5}
  m(E_{m_0,\varepsilon})>m(\Omega)-\frac{\varepsilon}{R}
\end{equation}
But for all $\varepsilon>0$, there is $\eta=\eta(\varepsilon)>0$, such that $m(T)<\eta$ implies
\begin{equation}\label{1.6}
  2^{p'^+}\int_T 2|k(x)|^{p'(x)}+c'(1+2^{r^+})|u_0(x)|^{r(x)}<\frac{\varepsilon'}{3},
\end{equation}
due to $k\in L^{p'(x)}(\Omega)$ and $u_0\in L^{r(x)}(\Omega)$, where $r(x)=(q(x)-1)p'(x),$\\
$c'=\max\{c^{p'^+},c^{p'^-}\}$ and $\varepsilon'=\inf\{\varepsilon^{p'^-},\varepsilon^{p'^+}\}.$
Now let
%\begin{equation}\label{1.7}
%  \eta_0=\eta(\frac{\varepsilon}{3}),
%\end{equation}
\begin{equation}\label{1.8}
  0<\delta<\min\{1,\frac{1}{m_0}(\frac{\eta}{2})^{\frac{1}{p^-}},(\frac{\varepsilon'}{3c'2^{p'^++r^+}})^{\frac{1}{r^-}}\},
\end{equation}
\begin{equation}\label{1.9}
  R>\max\{\varepsilon,\frac{2\varepsilon}{\eta},(\frac{3m(\Omega)}{\varepsilon'})^{\frac{1}{p'^-}}\varepsilon\}.
\end{equation}
Suppose that $\|u-u_0\|_{p(x)}<\delta$, and consider the set $E =\{x\in \Omega; |u(x)-u_0(x)|\geq\frac{1}{m_0}\}$; we have
\begin{equation}\label{1.10}
  m(E)<(m_0\delta)^{p^-}<\frac{\eta}{2}
\end{equation}
If $x\in E_{m_0,\varepsilon}\setminus E$, then, for each $w\in \phi u$,
$$|u(x)-u_0(x)|<\frac{1}{m_0}$$
and
$$w(x)\in (\underline{g}(x,u_0(x)-\frac{\varepsilon}{R},\overline{g}(x,u_0(x))+\frac{\varepsilon}{R}).$$
Let
\begin{eqnarray*}
  G^+ &=& \{x\in \Omega; w(x)>\overline{g}(x,u_0(x))\}, \\
  G^- &=&\{x\in \Omega; w(x)<\underline{g}(x,u_0(x))\}, \\
  G^0 &=& \{x\in \Omega; w(x)\in [\underline{g}(x,u_0(x),\overline{g}(x,u_0(x))].
\end{eqnarray*}
and
\begin{eqnarray*}
  y(x) &=& \left\{
             \begin{array}{ll}
               \overline{g}(x,u_0(x)), & \hbox{for } x\in G^+ ; \\
               w(x), & \hbox{for } x\in G^0; \\
               \underline{g}(x,u_0(x), & \hbox{for } x\in G^-.
             \end{array}
           \right.
\end{eqnarray*}
Then $y\in \phi u_0$ and
\begin{equation}\label{1.11}
  |y(x)-w(x)|<\frac{\varepsilon}{R} \mbox{ for all } x\in E_{m_0,\varepsilon}\setminus E.
\end{equation}
Combining \eqref{1.9} with \eqref{1.11}, we obtain
\begin{equation}\label{1.12}
  \int_{E_{m_0,\varepsilon}\setminus E}|y(x)-w(x)|^{p'(x)}\;dx<(\frac{\varepsilon}{R})^{p'^-}m(\Omega)<\frac{\varepsilon'}{3}.
\end{equation}
Let $V$ be the coset in $\Omega$ of $E_{m_0,\varepsilon}\setminus E$; then $V =(\Omega\setminus E_{m_0,\varepsilon})\cup(E_{m_0,\varepsilon}\cap E)$ and
$$m(V)\leq m(\Omega\setminus E_{m_0,\varepsilon})+m(E_{m_0,\varepsilon}\cap E)<\frac{\varepsilon}{R}+m(E)<\eta,$$
in view of \eqref{1.5}, \eqref{1.10} and \eqref{1.9}. Combining $(g_2)$ and \eqref{1.6} with \eqref{1.8}, we obtain
\begin{eqnarray*}
  \int_V|y(x)-w(x)|^{p'(x)}\;dx &\leq& \int_V(|y(x)|^{p'(x)}+|w(x)|^{p'(x)})\;dx \\
   &\leq& 2^{p'^+}\int_V(|b(x)|^{p'(x)}+c^{p'(x)}|u_0(x)|^{r(x)} \\
   &+&|b(x)|^{p'(x)}+c^{p'(x)}+|u(x)|^{r(x)})\;dx \\
   &\leq& 2^{p'^+}\int_V(2|b(x)|^{p'(x)}+c^{p'(x)}|u_0(x)|^{r(x)} \\
   &+&c^{p'(x)}+2^{r(x)}(|u(x)-u_0(x)|^{r(x)}+|u_0(x)|^{r(x)}))\;dx \\
   &\leq& 2^{p'^+}\int_V(2|b(x)|^{p'(x)}+c'(1+2^{r(x)})|u_0(x)|^{r(x)})\;dx \\
   &+& 2^{p'^+}c'\int_\Omega 2^{r(x)}|u(x)-u_0(x)|^{r(x)}\;dx \\
  &\leq& \frac{\varepsilon'}{3}+\frac{\varepsilon'}{3},
\end{eqnarray*}
Thus
\begin{equation}\label{1.13}
  \int_V|y(x)-w(x)|^{p'(x)}\;dx\leq 2\frac{\varepsilon'}{3}.
\end{equation}
Combining \eqref{1.12} with \eqref{1.13}, we see that $$\rho_{p'(x)}(y-w)<\varepsilon'.$$
- If $\varepsilon\geq1$, then $\varepsilon'=\varepsilon^{p'^-}$. From \eqref{imp1} and \eqref{imp2}, we have
$$\|y-w\|_{p'(x)}^{p'^-}<\varepsilon^{p'^-} \mbox{ or } \|y-w\|_{p'(x)}^{p'^+}<\varepsilon^{p'^-},$$
then $$\|y-w\|_{p'(x)}<\varepsilon \mbox{ or } \|y-w\|_{p'(x)}<\varepsilon^{\frac{p'^-}{p'^+}}\leq\varepsilon.$$
- If $\varepsilon<1$, then $\varepsilon'=\varepsilon^{p'^+}$. From \eqref{imp0} and \eqref{imp2}, we have
$\|y-w\|_{p'(x)}^{p'^+}<\varepsilon^{p'^+},$
then $$\|y-w\|_{p'(x)}<\varepsilon.$$
Therefore  $\phi$ is u.s.c.\\
Hence $I^*o\phi oI$ is obviously bounded, u.s.c and compact.
\end{proof}

\section{Main result}
\begin{definition}
  We call that $u\in W_0^{1,p(x)}(\Omega)$ is a weak solution of (\ref{Pr1}) if there exists $z\in Nu$ such that
  $$\int_\Omega|\nabla u|^{p(x)-2}\nabla u\nabla v\;dx+\int_\Omega uv\;dx +\langle z,v\rangle= 0, \;\;\forall v\in W_0^{1,p(x)}(\Omega).$$
\end{definition}
\begin{theorem}
  Under assumptions $(g_1)$ and $(g_2)$, the problem (\ref{Pr1}) has a weak solution $u$ in $W_0^{1,p(x)}(\Omega).$
\end{theorem}
\begin{proof}
Let $L, A :W_0^{1,p(x)}(\Omega)\rightarrow W^{-1,p'(x)}(\Omega)$ and $N:W_0^{1,p(x)}(\Omega)\rightarrow 2^{W^{-1,p'(x)}(\Omega)}$ be definded in Lemmas \ref{L}, \ref{A} and \ref{N}, respectively. Then $u\in W_0^{1,p(x)}(\Omega)$ is a weak solution of \eqref{Pr1} if and only if
  \begin{equation}\label{eq2}
    Lu\in-(A+N)u.
  \end{equation}
Thanks to the properties of the operator $L$ seen in Lemma \ref{L} and in view of Minty-Browder Theorem (see \cite{Z}, Theorem 26A), the inverse operator \\$T:=L^{-1}:W^{-1,p'(x)}(\Omega)\rightarrow W_0^{1,p(x)}(\Omega)$ is bounded, continuous and satisfies condition $(S_+)$. Moreover, note by Lemmas \ref{A} and \ref{N} that the operator\\
$S:=A+N:W_0^{1,p(x)}(\Omega)\rightarrow 2^{W^{-1,p'(x)}(\Omega)}$ is bounded, u.s.c and quasimonotone. Consequently, equation (\ref{eq2}) is equivalent to
\begin{equation}\label{eq3}
  u=Tv \mbox{ and } v\in -STv.
\end{equation}
To solve equation (\ref{eq3}), we will apply the degree theory introducing in section 2. To do this, we first claim that the set
$$B:=\{v\in W^{-1,p'(x)}(\Omega)|v\in -tSTv \mbox{ for some } t\in[0,1]\}$$ is bounded. Indeed, let $v\in B$, that is $v+ta=0$ for some $t\in[0,1]$ where $a\in STv.$ Set $u:=Tv$, then $|Tv|_{1,p(x)}=|\nabla u|_{p(x)}.$ We write $a=Au+z\in Su,$ where $z\in Nu,$ that is $\langle z,u\rangle=\int_\Omega wu\;dx,$ for some $w\in L^{p'(x)}(\Omega)$ with $\underline{g}(x,u(x))\leq w(x)\leq\overline{g}(x,u(x))$ for a.e. $x\in\Omega.$\\
If $\|\nabla u\|_{p(x)}\leq 1$, then $\|Tv\|_{1,p(x)}$ is bounded.\\
If $\|\nabla u\|_{p(x)}>1$, then we get by the implication \eqref{imp1}, the growth condition $(g_2)$, the H\"older inequality \eqref{hol} and the inequality \eqref{ineq1} the estimate
\begin{eqnarray*}
  \|Tv\|_{1,p(x)}^{p^-} &=& \|\nabla u\|_{p(x)}^{p-} \\
   &\leq& \rho_{p(x)}(\nabla u) \\
   &=& \langle Lu,u\rangle  \\
  &=& \langle v,Tv\rangle \\
   &=& -t\langle a,Tv \rangle  \\
   &=& -t\int_\Omega (u+w)u\; dx \\
  &\leq& const(+\int_\Omega|u|^2\;dx\int_\Omega|k(x)u(x)|\;dx+\rho_{q(x)}(u)) \\
  &\leq& const(\|u\|_{L^2}^2+\|k\|_{p'(x)}\|u\|_{p(x)}+\|u\|_{q(x)}^{q^+}+\|u\|_{q(x)}^{q^-})\\
  &\leq& const(\|u\|_{L^2}^2+\|u\|_{p(x)}+\|u\|_{q(x)}^{q^+}+\|u\|_{q(x)}^{q^-}).
\end{eqnarray*}
From the Poincar\'e inequality (\ref{ptcar}) and the continuous embedding $L^{p(x)}(\Omega)\hookrightarrow L^2(\Omega)$ $L^{p(x)}(\Omega)\hookrightarrow L^{q(x)}(\Omega)$, we can deduct the estimate
$$\|Tv\|_{1,p(x)}^{p^-} \leq const(\|Tv\|_{1,p(x)}^2+\|Tv\|_{1,p(x)}+\|Tv\|_{1,p(x)}^{q^+}).$$
It follows that $\{Tv|v\in B\}$ is bounded.\\
Since the operator $S$ is bounded, it is obvious from \eqref{eq3} that the set $B$ is bounded in $W^{-1,p'(x)}(\Omega)$. Consequently, there exists $R>0$ such that
$$\|v\|_{-1,p'(x)}<R \mbox{ for all } v\in B.$$
This says that
$$v\notin -tSTv, \mbox{ for all } v\in\partial B_R(0) \mbox{ and all } t\in[0,1].$$
From Lemma (\ref{l2.1}) it follows that
$$I+ST\in\mathcal{F}_T(\overline{B_R(0)})\mbox{ and } I=LT\in\mathcal{F}_T(\overline{B_R(0)}).$$
Consider a homotopy $H:[0,1]\times\overline{B_R(0)}\rightarrow 2^{W^{-1,p'(x)}(\Omega)}$ given by
$$H(t,v):=(1-t)Iv+t(I+ST)v \mbox{ for } (t,v)\in[0,1]\times\overline{B_R(0)}.$$
Applying the homotopy invariance and normalization property of the degree $d$ stated in Theorem(\ref{t2.1}), we get
$$d(I+ST,B_R(0),0)=d(I,B_R(0),0)=1,$$
and hence there exists a point $v\in B_R(0)$ such that
$$v\in -STv,$$
which says that $u=Tv$ is a solution of \eqref{eq2}. We conclude that $u=Tv$ is a weak solution of \eqref{Pr1}. This completes the proof.
\end{proof}

%\begin{acknowledgements}
%The author would like to thank the anonymous referees for their constructive comments and suggestions.
%\end{acknowledgements}

\end{document}